\documentclass[12pt, english, a4paper]{amsart}
\usepackage[margin=2.3cm]{geometry}
\usepackage{amssymb,amsmath,amsthm}
\usepackage{commath}
\usepackage{mathtools}
\usepackage{mathrsfs}
\usepackage{accents}
\usepackage[utf8]{inputenc}
\usepackage[T1]{fontenc}

\usepackage{caption}

\usepackage{xfrac}
\usepackage{enumerate}
\usepackage{graphics, float}

\usepackage[colorlinks]{hyperref}
\hypersetup{allcolors = black}

\usepackage{url}
\usepackage[T1]{fontenc} 
\usepackage{fourier}
\usepackage{bm}
\usepackage{algorithm}
\usepackage{algpseudocode}
\usepackage{datetime}

\theoremstyle{plain}
\newtheorem{theorem}{Theorem}
\newtheorem{conj}[theorem]{Conjecture}

\newtheorem{definition}[theorem]{Definition}

\newcommand{\R}{\mathbb{R}}

\newcommand{\eps}{\varepsilon}

\newcommand{\K}{\mathcal{K}}

\makeatletter
\@namedef{subjclassname@2020}{
  \textup{2020} Mathematics Subject Classification}
\makeatother

\usepackage[displaymath,textmath,graphics, subfigure, floats]{preview}
\PreviewEnvironment{center} 
\PreviewEnvironment{pspicture} 

\makeatother

\title{A note on the Steinitz Lemma}
\author{Gergely Ambrus}

\author{Rainie Heck}

\thanks{Research of G. A. was partially supported by the ERC Advanced Grant "GeoScape" no.  882971, by Hungarian National Research (NKFIH) grants no. KKP-133819, 147145, 147544, and 150151, which has been implemented with the support provided by the Ministry of Culture and Innovation of Hungary from the National Research, Development and Innovation Fund, financed under the ADVANCED-24 funding scheme. This research was funded by the grant 2024-1.2.8-TÉT-IPARI-CN-2025-00011,
with the support provided by the National Research,
Development and Innovation Office from the National Research,
Development and Innovation Fund, and financed under the
2024-1.2.8-TÉT-IPARI-CN funding scheme. Research of R. H. is supported by the National Science Foundation Graduate Research Fellowship under Grant No. DGE-2140004. Research was supported by project no. TKP2021-NVA-09, which has been implemented with the support provided by the
Ministry of Innovation and Technology of Hungary from the National
Research, Development and Innovation Fund, financed under the
TKP2021-NVA funding scheme. 
}

\begin{document}

\begin{abstract}
We establish the connection between the Steinitz problem for ordering vector families in arbitrary norms and its variant for not necessarily zero-sum families consisting of `nearly unit' vectors.
\end{abstract}


\maketitle

\section{Framework and results}\label{sec:intro}

More than a century ago, in 1913, Steinitz~\cite{Steinitz1913} proved the following unexpected result.

\begin{theorem}[Euclidean Steinitz lemma]\label{thm:EuclSteinitz}
Given any finite family of vectors $V\subset \R^d$ of Euclidean norm at most 1 that sum to 0, one can order the elements of $V$ as $v_1, \ldots, v_n$ so that for every $k = 1, \ldots, n,$ 
\begin{equation}\label{eq:steinitz_criterion}
\left|  \sum_{i=1}^k v_i\right| \leq C
\end{equation}
where $C$ is a constant that depends only on the dimension $d$. 
\end{theorem}
Here and later on, $|.|$ denotes the Euclidean norm on $\R^d$. 
Steinitz proved the statement with the constant $C = 2 d$; we will later see that stronger estimates exist. 
It is natural to ask for the smallest value of $C$ for which \eqref{eq:steinitz_criterion} is guaranteed to hold; this will be the focus of the paper. 

Let us introduce some notations and definitions.  For a finite family of vectors $V$, $\Sigma(V):=\Sigma_{v \in V}v$ will stand for the sum of the elements of $V$.  For $n \in \mathbb{N}_+$, we set $[n]:=\{ 1, \ldots, n \}$.  $B_p^d$ will denote the unit ball of the $\ell_p$-norm $\|.\|_p$ on $\R^d$ for $1 \leq p \leq \infty$. Let $\K_o^d$ denote the class of convex bodies in $\mathbb{R}^d$ (i.e., compact convex sets with non-empty interior) that contain the origin in their interior. For $B \in \K_o^d$, the Minkowski norm generated by $B$ is defined as
\begin{equation*}
    \|x\|_B:=\inf\{r\geq 0:\ x\in rB\}.
\end{equation*}
Note that this is a norm on $\R^d$ in the classical sense only when $B$ is symmetric about $0$, that is $B=-B$; otherwise, $\|.\|_B$ is homogeneous only for positive scalars -- in this case $\|.\|_B$ is called an asymmetric norm. The term `norm' will be used in this paper in a general sense that encompasses both cases. A convex body $K$ will be called symmetric if it is symmetric about the origin.

\begin{definition}[Steinitz constant] \label{def:steinitzconstant}
   Let $B\in \mathcal{K}_0^d$. The \emph{Steinitz constant of} $B$, denoted $S(B)$, is the smallest number $C$ for which any finite family of vectors $V\subset B$ with $\Sigma(V)=0$ has an ordering $V=\{ v_1, \ldots, v_n \}$ along which each partial sum has norm at most~$C$. That is, for every $k \in [n]$, 
   \begin{equation}\label{sumbound}
    \Big\| \sum_{i \in [k]}  v_i  \Big\|_B \leq C.
   \end{equation}
\end{definition}

We remark that the term `constant' above refers to the fact that  $S(B)$ depends only on the choice of $B$, but not on the vector family $V \subset B$. We make no reference to the dimension $d$, as the value of the Steinitz constant is independent of $d$ as long as $B$ can be embedded in $\R^d$.

We also consider an expanded variant of the Steinitz constant, where the zero-sum condition $\Sigma(V)=0$ on the vector family is dropped:

\begin{definition}[Relaxed Steinitz constant] \label{def:relaxedsteinitzconstant}
   For $B\in \mathcal{K}_0^d$, let $S^*(B)$ denote the smallest constant $C$ for which any finite family of vectors $V \subset B$  has an ordering $V=\{ v_1, \ldots, v_n \}$ so that 
   \begin{equation}\label{epssumbound}
    \Big\| \sum_{i \in [k]}  v_i  - \frac k n \Sigma(V) \Big\|_B \leq C
   \end{equation}
holds for every $k \in [n]$.
\end{definition}

The relationship with the original Steinitz constant is given by the simple chain of inequalities
\begin{equation}\label{relax-orig-Steinitz}
    S(B) \leq S^*(B) \leq (1 + \rho(B)) S(B)
\end{equation}
where 
\begin{equation}\label{eq:rhoB}
\rho(B) := \max_{v \in B} \| -v \|_B
\end{equation}
measures the asymmetry of $B$. Note that $\rho(B) = 1$ if $B$ is symmetric.
The lower bound in \eqref{relax-orig-Steinitz} is trivial; to see the upper estimate, one has to observe that starting from any family $V$ of $n$ vectors in $B$, the triangle inequality implies that $\| \Sigma(V) \|_B \leq n$, hence $\| - \frac{\Sigma(V)}{n} \|_B \leq \rho(B)$. Accordingly, the zero-sum vector family $\big \{v -\frac{\Sigma(V)}{n}  : v\in V \big \}$ lies in $(1 + \rho(B))B$, and the estimate readily follows.

We note that there are further variants of Definition~\ref{def:relaxedsteinitzconstant} (see e.g. \cite{banaszczyk1990note}), although these are not directly related to our subsequent results.

Theorem~\ref{thm:EuclSteinitz}, proved by Steinitz, justifies that $S(B_2^d)$ and, via \eqref{relax-orig-Steinitz}, that $S^*(B_2^d)$ are well defined. The proof can be extended to any symmetric norm. For asymmetric norms, the justification of Definitions \ref{def:steinitzconstant} and \ref{def:relaxedsteinitzconstant} is implied by \eqref{relax-orig-Steinitz} and the following general bound, which is a direct corollary of the results obtained for arbitrary $B$-norms in 1978 by Sevastyanov \cite{SevastyanovSteinitz} and in 1980 by Grinberg and Sevastyanov \cite{grinberg_sevastyanov_ValueOfSteinitzConstant}.

\begin{theorem}[The Steinitz Lemma for general norms \cite{grinberg_sevastyanov_ValueOfSteinitzConstant,SevastyanovSteinitz}] 
\label{thm:GeneralSteinitz}
For any convex body $B \in \K_o^d$,
\begin{equation}\label{eq:steinitzgeneral}
    S(B)\leq d.
\end{equation}    
\end{theorem}

The bound is tight for non-symmetric convex bodies, as is shown by taking $B$ to be the regular simplex centered at the origin and choosing $V$ to be the set of its vertices, whereas it is sharp by the order of magnitude for symmetric norms, which is confirmed by the inequality $S(B_1^d)\geq (d+1)/2$, see~\cite{grinberg_sevastyanov_ValueOfSteinitzConstant}. An example in the Euclidean case~\cite{damsteeg1950steinitz, grinberg_sevastyanov_ValueOfSteinitzConstant}  shows that $S(B_2^d) \geq \sqrt{d +3}/2$ must hold. For symmetric $B \in K_o^d$, the estimate in \eqref{eq:steinitzgeneral} can be strengthened to $d - 1 + \frac 1 d$, see \cite{Sevastyanov91}.

We note that more general versions of Theorem~\ref{thm:GeneralSteinitz} were proved 
for convex sets not necessarily containing the origin, without assuming the zero sum condition 
\cite{grinberg_sevastyanov_ValueOfSteinitzConstant,Sevastyanov80}.

Via \eqref{relax-orig-Steinitz}, Theorem~\ref{thm:GeneralSteinitz} readily implies the bound
\begin{equation*}
    S^*(B)\leq (1 + \rho(B))d, 
\end{equation*}
which also follows from the results in~\cite{grinberg_sevastyanov_ValueOfSteinitzConstant}. In particular, $S^*(B) \leq 2 d$ holds for symmetric $B \in \K_o^d$.

The following long-standing conjecture of Bergström~\cite{bergstrom1931Zwei}, be it confirmed, would yield a much stronger estimate on the Steinitz constant in the Euclidean case:

\begin{conj}\label{SteinitzEucConj}
    For all $d\geq 1$, $S(B_2^d)=O(\sqrt{d})$.
\end{conj}

\noindent
We note that a matching bound is also expected to hold for the maximum norm $\|.\|_\infty$ on $\R^d$. 

So far, Conjecture~\ref{SteinitzEucConj}, that is sometimes also called the Euclidean Steinitz problem, has refuted all attempts. Our aim in this short note is to reduce the investigation to families consisting of `nearly unit' vectors at the cost of an additive $O(1)$ error in the case of general norms.

To that end, we introduce the restricted variants of the Steinitz constants for `nearly unit' vectors: the subscript `$\eps$' will mean that only families of vectors are considered whose members have norm in the interval $[1 - \eps, 1]$.

\begin{definition}[$\eps$-Steinitz constants]\label{def:epsSteinitzconstant}
   For $B\in \mathcal{K}_0^d$ and $0 \leq \eps \leq 1$, let $S^*_\eps(B)$ denote the smallest constant $C$ for which any finite family $V \subset \R^d$ consisting of vectors of $\|.\|_B$-norm in $[1-\eps, 1]$ may be ordered as $V=\{ v_1, \ldots, v_n \}$ so that 
   \begin{equation*}
    \Big\| \sum_{i \in [k]}  v_i  - \frac k n \Sigma(V) \Big\|_B \leq C
   \end{equation*}
holds for every $k = 1, \ldots, n$. Furthermore, let $S_\eps(B)$ denote analogous quantity for vector families that satisfy the extra condition $\Sigma(V) = 0$.
\end{definition}

\noindent
Note that for any $0 \leq \eps \leq 1$,  $S_0(B) \leq S_\eps(B) \leq S_1(B) = S(B)$,  $S^*_0(B) \leq S^*_\eps(B) \leq S^*_1(B)= S^*(B)$, and $S_\eps(B)\leq  S^*_\eps(B)$. Thus, \eqref{relax-orig-Steinitz} ensures that 
\begin{equation}\label{epssteinitz_orig}
    S^*_\eps(B)\leq 2 S(B)
\end{equation}
for symmetric norms, while 
\[
S^*_\eps(B)\leq (1 + \rho(B)) S(B)
\]
holds for arbitrary $B \in \K_o^d$. 

Furthermore, observe that setting $\eps = 0$ restricts the problem to families of unit vectors. In the Euclidean case, a construction given by Damsteeg and Halperin~\cite{damsteeg1950steinitz} implies that 
\begin{equation} \label{eq:S0}
  \Omega(\sqrt{d}) \leq S_0(B_2^d) \leq S^*_0(B_2^d) \leq S^*_\eps(B_2^d).
\end{equation}

Our result establishes a reverse estimate of \eqref{epssteinitz_orig} that holds for any norm.

\begin{theorem}\label{general theorem}
For all $d\geq 2$, any convex body $B \in \K_o^d$, and $0<\varepsilon \leq 1$,
\begin{equation}\label{SteinitzBound}
        S(B)< \frac{1}{\varepsilon}\Big(S^*_\eps(B)+2\rho(B)+1\Big).
    \end{equation}
\end{theorem}

In the case that $B$ is symmetric, the bound simplifies to $\frac{1}{\varepsilon}(S_\varepsilon^*(B)+3)$. In particular, an $o(d)$ bound on $S^*_\eps(B_2^d)$ for some fixed $0 < \eps \leq 1$
would yield an $o(d)$ estimate on $S(B_2^d)$, hence improving the currently strongest bound. Moreover, \eqref{eq:S0} and \eqref{SteinitzBound} imply that Conjecture~\ref{SteinitzEucConj} is equivalent to the statement that $S^*_\eps(B_2^d) = O(\sqrt{d})$ for some constant $\eps \in (0,1]$.

\section{Excerpts from the History of the Steinitz Lemma}\label{history}

Despite being more than a century old, the story of the Steinitz lemma is still far from complete. In the following, we provide an overview that lists some of the main, and often forgotten, steps in its development.

Every mathematician knows the Riemann series theorem, dating more than 150 years ago~\cite{RiemannRearrangementThm}, which states that given a conditionally convergent series, one can rearrange it to converge to any real number. In other words, for any real series, the set of all sums of its rearrangements is either empty (meaning that the series is divergent), a single point (the series is absolutely convergent), or the entire real line. 
It is natural to ask for the extension of this result to series consisting of complex numbers and, even more generally, of vectors in $\R^d$. In his very first article in 1905, at the age of 19, Lévy~\cite{levy1905series} addressed this question and announced the following result.

\begin{theorem}[Lévy-Steinitz theorem]\label{LevySteinitz}
Given a series of vectors in $\R^d$, the set of all sums of its rearrangements is either empty, or it forms an affine subspace of $\ \R^d$. 
\end{theorem}

Recall that affine subspaces of $\R^d$ are of the form $L + x$ where $L \subseteq \R^d$ is a linear subspace and $x \in \R^d$.

Unfortunately, the proof of Lévy contained serious gaps for dimensions $d \geq 3$, as pointed out by Steinitz~\cite{Steinitz1913}  in 1913. In turn, he gave the first complete proof of Theorem~\ref{LevySteinitz}, known today as the Lévy-Steinitz theorem.  
Steinitz's work is quite technical and wide-scoped: it was published in three parts ~\cite{Steinitz1913, Steinitz1914, Steinitz1916}, with total length summing well over 100 pages. A key step of his proof is Theorem~\ref{thm:EuclSteinitz} (see \cite[p.171]{Steinitz1913}), which he stated with $C = 2d$.

Independently of Lévy and Steinitz, Gross~\cite{Gross1917} also found a shorter proof for Theorem~\ref{LevySteinitz} that is reminiscent of Steinitz's method. His approach is again based on the rediscovered Steinitz lemma, yet it yields only the weaker constant $S(B_2^d) \leq 2^d -1$, which is an inevitable consequence of the induction dimension technique he applies. He also provides a geometric reformulation of Theorem~\ref{thm:EuclSteinitz}: given any closed polygonal path in $\R^d$ starting at the origin with side lengths not exceeding 1, it is possible to rearrange the order of its sides so that the resulting polygonal path does not leave the ball of radius $C$. This later led to the alternate title ``polygonal confinement theorem'' for Theorem~\ref{thm:EuclSteinitz} (cf. Rosenthal~\cite{rosenthal1987remarkable}).

Gross was not the last one to rediscover the Steinitz lemma. In 1931, Bergström published two papers on the topic. In the first \cite{bergstrom1931neuer}, he gives an alternative proof for Theorem~\ref{LevySteinitz}. The crux of his proof is again Theorem~\ref{thm:EuclSteinitz}, which he considers to be of interest on its own, and proves by induction on the dimension, leading to the estimate $S(B_2^d) \leq \sqrt{(4^d -1)/3}$. His second article~\cite{bergstrom1931Zwei} concentrates solely on the Steinitz lemma. He proves that any family of vectors $V \subset B_2^2$ with $\Sigma(V)=0$ can be arranged to form a closed polygonal path that fits in a circle of radius $\sqrt{5}/2$ (not necessarily centered at the origin), leading  to the upper bound $S(B_2^2) \leq  (\sqrt{5}+1)/2$ in the plane. Regarding the question in arbitrary dimensions, he formulates Conjecture~\ref{SteinitzEucConj}.

In 1936, Hadwiger \cite{hadwiger1936satz} became the first one to study the Steinitz lemma (which he attributes to Gross and Bergström) for series in general inner product spaces: he manages to bound the norm of partial sums in terms of the number of vectors $n$. His attention then turned to the extension of the Lévy-Steinitz theorem to abstract Hilbert spaces~\cite{hadwiger1940umordnungsproblem} and finite-dimensional vector spaces~\cite{hadwiger1942satz}.

Returning to the Euclidean case, Damsteeg and Halperin~\cite{damsteeg1950steinitz} provided a construction of Euclidean unit vectors establishing $\frac 1 2 \sqrt{d + 3} \leq S_0(B_2^d) \leq S(B_2^d)$, which implies that the $O(\sqrt{d})$ bound conjectured by Bergström would be  optimal by the order of magnitude (this is also shown by considering the vertex set of a centered regular simplex). 

In 1954, Behrend entered the scene \cite{behrend1954steinitz} and by a refinement of Steinitz's original method strengthened the estimate to $S(B_2^d) < d$ for every $d \geq 3$. He also showed that $S(B_2^3) \leq \sqrt{5 + 2 \sqrt{3}}$. 

Much of the above information had been blocked by the iron curtain. Although Theorem~\ref{thm:EuclSteinitz} is noted to be a `well-known lemma of Steinitz'~\cite{grinberg_sevastyanov_ValueOfSteinitzConstant},  Kadets~\cite{Kadets} only rediscovered Bergström's estimate $S(B_2^d) \leq \sqrt{(4^d-1)/3}$ in 1953.
Twenty years later, Sevastyanov studied several variants of the question, and introduced the compact vector summation problem. In 1973, he rediscovered and re-proved the planar case of the Steinitz lemma with the bound $S(B_2^2) \leq \sqrt{3}$, see~\cite{Sevastyanov75}. In the higher dimensional case, he gave an efficient algorithm for ordering a vector family as in Definition~\ref{def:steinitzconstant} so as to achieve the bound in Theorem~\ref{thm:GeneralSteinitz}. This result extended Behrend's bound up to losing the strict inequality. Two years later, a joint paper of Grinberg and Sevastyanov \cite{grinberg_sevastyanov_ValueOfSteinitzConstant} considered a more general problem of ordering vectors from a given family $V\subset B\in \K_o^d$ with an arbitrary sum $\Sigma(V)$. Their result implied Theorem \ref{thm:GeneralSteinitz} as a corollary. Moreover, they showed that the bound \eqref{eq:steinitzgeneral} on the Steinitz constant is tight on the class of asymmetric norms.
    In \cite{Sevastyanov91}, Sevastyanov further strengthened the estimate to $S(B) \leq d - 1 + \frac 1 d$ (thus confirming Banaszczyk's claim \cite{Banaszczyk1987}). 
For further developments related to  algorithmic aspects, see~\cite{SevSurvey1998,SevSurvey94}.

Meanwhile in Hungary, independently of the work in the USSR, Fiala~\cite{FialaSteinitz} also rediscovered (after Sevastyanov~\cite{Sevastyanov74} and Belov and Stolin~\cite{BelovStolin}) the connection between the flow shop problem and the Steinitz lemma, and re-proved the latter in the planar case. Inspired by his work, Bárány~\cite{barany1981vector} proved the bound $S(B)\leq 3d/2$ for symmetric $B$'s (note that this is weaker than the estimate in~\cite{grinberg_sevastyanov_ValueOfSteinitzConstant, SevastyanovSteinitz}).

Still in the 1980s, Halperin \cite{halperin1986sums} applied a variation of Lévy's method to obtain an elementary proof of Theorem~\ref{LevySteinitz}, using the Grinberg-Sevastyanov variant of the Steinitz lemma, called here as the `Polygon Rearrangement Theorem'. He also studies the question in $L_p$ and $\ell_p$ spaces for $0 < p \leq \infty$. In an expository article, Rosenthal~\cite{rosenthal1987remarkable} presented the Lévy-Steinitz theorem along the Gross-Steinitz approach.  He cites the Steinitz lemma as the `Polygonal Confinement Theorem' with the weak bound $S(B_2^d) \leq O(2^d)$, apparently unaware of its stronger forms.

Concentrating on the planar case, Banaszyczyk showed in~\cite{Banaszczyk1987} that $S(B) \leq 3/2$ for any symmetric $B \subset \R^2$, and this bound is achieved when $B$ is a square centered at the origin. In~\cite{banaszczyk1990note}, he determined the exact value of the planar Euclidean Steinitz constant:   $S(B_2^2) =\sqrt{5}/2$.

A possible approach for attacking Conjecture~\ref{SteinitzEucConj} is via Chobanyan's transference theorem, which gives an explicit connection between the Steinitz constant and the sign-sequence constant $E(B)$, where one is asked to assign signs to vectors of a (potentially infinite) sequence of vectors selected from $B$ so that all partial sums are bounded by $E(B)$. Chobanyan's result~\cite{chobanyan, chobanyan1987structure} shows that $S(B)\leq E(B)$. For further information about the sign-sequence constant, see the survey article of Bárány~\cite{Ba10}.

This brief summary is too short for listing further related topics in detail, such as coordinate-dependent Steinitz bounds for the maximum norm (see e.g.~\cite{BanSevMaxNorm}), various extensions of the Lévy-Steinitz theorem to infinite-dimensional spaces (see e.g. \cite{katznelson1974conditionally}), or colorful versions of the Steinitz lemma (see~\cite{barany2024matrix, oertel2024colorful}). Questions on vector sums have a vast literature: the bibliography of Halperin and Ando~\cite{halperin1989bibliography} lists 257 works up to 1988, with at least a hundred related articles having been published since then.

\section{Proof of Theorem \ref{general theorem}}

We will use the following notation: for any unit vector $u\in S^{d-1}$ (the Euclidean unit sphere in $\R^d$), the closed positive halfspace orthogonal to $u$ is
\begin{equation*}
    H_+(u):=\{x\in\mathbb{R}^d:\ \langle x, u\rangle \geq 0\}.
\end{equation*}

\begin{proof}[Proof of Theorem \ref{general theorem}]
Fix $0<\varepsilon<1$, $d\geq 2$, and  $B \in \K_o^d$, and suppose that we are given a finite vector family $V\subset B$ with $\Sigma(V)=0$; our goal is to order $V$ in such a way that all partial sums are bounded by the right-hand side of \eqref{SteinitzBound}. For simplicity, we will write $\|.\| := \|.\|_B$ throughout the proof, omitting the dependence on $B$.

The first step is to partition $V$ as
\begin{equation*}
    V=\Big(\bigcup_{\alpha\in A} V_\alpha\Big)\ \bigcup\  R,
\end{equation*}
where  $A$ is an index set of cardinality $m$,
satisfying the following properties:
\begin{itemize}
    \item[(i)] For each $\alpha\in A$, there exists $u_\alpha\in S^{d-1}$ such that $V_\alpha\subset H_+(u_\alpha)$
    \item[(ii)] For each $\alpha\in A$,
    \begin{equation*}\label{sumbounds}
        \frac{1}{\varepsilon}-1\leq \|\Sigma(V_\alpha)\| < \frac 1 \eps.
    \end{equation*}
    \item[(iii)] For any $u\in S^{d-1}$, and any subset $T\subseteq R$,
    \begin{equation*}
        \big\|\Sigma(T\cap H_+(u))\big\|<1/\varepsilon.
    \end{equation*}
\end{itemize}

Note that we intentionally use an unordered index set $A$ of cardinality $m$, rather than $A=[m]$, to emphasize that the vectors are not yet ordered. 

We define the families $V_\alpha$ via the following process: initialize $R:=V$. As long as there exist $u\in S^{d-1}$ and $T\subseteq R$ such that
\begin{equation}\label{sumnorm}
    \Big\|\Sigma(T\cap H_+(u))\Big\|\geq \frac{1}{\varepsilon}-1,
\end{equation}
we set $u_\alpha := u$ and select a containment-wise minimal subfamilies $V_\alpha\subseteq T\cap H_+(u_\alpha)$ so that 
\begin{equation}\label{eq:mincrit}
 \|\Sigma(V_\alpha)\| \geq \frac{1}{\varepsilon}-1
\end{equation}
holds. By the minimality condition, the triangle inequality guarantees that $\|\Sigma(V_\alpha)\| <  \frac 1 \eps$, therefore $V_\alpha$ satisfies property (ii).
We then update $R:=V\setminus V_\alpha$ and proceed until there is no choice of a vector $u\in S^{d-1}$ and $T\subseteq R$ satisfying \eqref{sumnorm}. 

It is an immediate consequence of the construction that properties (i)--(iii) are satisfied for the partition of $V$ obtained above.

Next, for every $\alpha\in A$, let $w_\alpha:=\varepsilon\cdot \Sigma(V_\alpha)$, and define 
$W:=\{w_\alpha\}_{\alpha\in A}$. Property (ii) yields that for every $ \alpha\in A$,
\begin{equation*}
    1-\varepsilon\leq \|w_\alpha\|< 1.
\end{equation*}
Thus, by the definition of $S^*_\eps(B)$, there exists an ordering $w_1,...,w_m$ of $W$ such that for any $j\in[m]$, 
\begin{equation}\label{Worder}
    \Big\|\sum_{i\in[j]}w_i-\frac{j}{m}\Sigma(W)\Big\|\leq S^*_\eps(B).
\end{equation}
Further, as $\Sigma(V)=0$ and $V=(\bigcup_{\alpha\in A} V_\alpha)\bigcup R$, 
\begin{equation*}\label{Wsum}
    \Sigma(W)=\varepsilon\sum_{\alpha\in A}\Sigma(V_\alpha)=-\varepsilon\Sigma(R),
\end{equation*}
and combinig this with \eqref{Worder} yields that for any $j \in [m]$,
\begin{equation}\label{Wbound}
    \Big\|\sum_{i\in[j]}w_i+\varepsilon \frac{j}{m}\Sigma(R)\Big\|\leq S^*_\eps(B).
\end{equation}

We now order our original collection of vectors $V$ as follows: fix the ordering of $W$ as above so that \eqref{Worder} holds, and note that this ordering of $W=\{w_\alpha\}_{\alpha\in A}$ induces a matching ordering $V_1,...,V_m$ of the families $\{V_\alpha\}_{\alpha\in A}$. Within each family $V_i$ for $i\in[m]$ we order the vectors $v\in V_i$ arbitrarily. Finally, we also order the remaining vectors in $R$ arbitrarily. Along this ordering, two types of partial sums occur:

\begin{itemize}
    \item[(a)] $\sum_{i\in[j]} \Sigma(V_i) + \Sigma(U)$ for $0 \leq j\leq m-1$ and $U\subset V_{j+1}$ (where $U=\emptyset$ is allowed);
    \item[(b)] $\sum_{i\in[m]}\Sigma(V_i)+\Sigma(T)$ with some $T\subseteq R$.
\end{itemize}
To prove \eqref{SteinitzBound}, we need to show that any partial sum of type (a) or (b) has norm at most $\frac{1}{\varepsilon}\Big(S^*_\eps(B)+2 \rho(B) +1\Big)$. 
Recall that by property (iii) of our construction, for any subset $T'\subseteq R$ and any $u\in S^{d-1}$,
\begin{equation*}
    \|\Sigma(T'\cap H_+(u))\|<1/\varepsilon.
\end{equation*}
Fix any direction $u\in S^{d-1}$, and partition $T'$ as
\begin{equation*}
    T'_+:=T'\cap H_+(u), \ \ \ T'_-:= T' \setminus T'_+ .
\end{equation*} 
Note that these are subfamilies of $R$, moreover, $T'_- \subset H_+(-u)$. Therefore, property (iii) implies that 
\begin{equation*}
    \|\Sigma(T'_+)\| < \frac 1 \eps \ \textrm{ and } \ \|\Sigma(T'_-)\|<\frac 1 \eps,
\end{equation*}
thus by the triangle inequality
\begin{equation}\label{halfspacepartition}
   \|\Sigma(T')\|< \frac 2 \varepsilon.
\end{equation}

We are ready to estimate the norm of partial sums along the ordering specified above. This is simple for sums of type (b): applying \eqref{halfspacepartition} for the family $T':=R\setminus T$,
\begin{equation*}
\Big\|\sum_{i\in[m]}\Sigma(V_i)+\Sigma(T)\Big\|=\Big\|-\Sigma(R) + \Sigma(T)\Big\|=\Big\| -\Sigma(R\setminus T)\Big\|  \leq \rho(B) \Big\| \Sigma(R\setminus T)\Big\|  < \rho(B) \frac{2}{\varepsilon}.
\end{equation*}
Finally, we handle the sums of type (a). Dividing both sides of \eqref{Wbound} by $\varepsilon$ and applying the triangle inequality yields that for any $0 \leq j\leq m-1$,
\begin{align*}
    \Big\|\sum_{i\in[j]} \Sigma(V_i) \Big\| &\leq 
    \Big\|-\frac{j}{m}\Sigma(R)\Big\|+ \Big\|\sum_{i\in[j]} \Sigma(V_i) + \frac{j}{m}\Sigma(R)
     \Big\|    \\
    &\leq
    \frac{j}{m}\|-\Sigma(R)\|+\frac{1}{\varepsilon}S^*_\eps(B)\\
    &<\rho(B)\frac{2}{\varepsilon}+\frac{1}{\varepsilon}S^*_\eps(B),
\end{align*}
where we have used $j< m$ and \eqref{halfspacepartition} with $T'=R$. Recall that $V_{j+1}$ was chosen as a minimal set (with respect to containment) that satisfies inequality \eqref{eq:mincrit}, that is, $\frac{1}{\varepsilon}-1\leq \|\Sigma(V_{j+1})\|$. In particular, for any $U\subset V_{j+1}$, we know that $\|\Sigma(U)\| <  \frac 1 \eps - 1 < \frac{1}{\varepsilon}$. Combining these estimates, for fixed $0\leq j\leq m-1$ and $U\subset V_{j+1}$ we conclude that 
\begin{equation*}
    \Big\|\sum_{i\in [j]} \Sigma(V_i)+\Sigma(U)\Big\|\leq \Big\|\sum_{i\in [j]} \Sigma(V_i)\Big\|+\Big\|\Sigma(U)\Big\|< \frac{1}{\varepsilon}\Big(S^*_\eps(B)+2\rho(B)+1\Big).
\end{equation*}

We have now shown that the $B$-norm of every partial sum along the specified ordering is strictly less than
\begin{equation*}
\frac{1}{\varepsilon}\left(S^*_\eps(B)+2\rho(B)+1\right).
\qedhere
\end{equation*}
\end{proof}

\noindent 
{\bf Remark.} 
The proof can not be transformed so as to provide an estimate on $S^*(B)$ instead of $S(B)$: since there is no upper bound on the size of the families $V_\alpha$, the size of these re-groupings need not be uniform. Hence, the average of the families $V_i$ and the whole family $V$ may differ drastically, which yields that the quantity in \eqref{epssumbound} cannot be estimated in terms of the individual deviations corresponding to $V_i$. 
\section{Acknowledgements}

We are indebted to Imre Bárány for the fruitful discussions and for providing us with reference~\cite{halperin1989bibliography}, and to the anonymous referees for their suggestions that improved the presentation and led to strengthening the estimate in Theorem 7 and to its extension
to arbitrary norms, as well as to a simplification of its proof.

\bibliographystyle{abbrv}
\bibliography{SteinitzBibliography}

\begin{thebibliography}{10}

\bibitem{Banaszczyk1987}
W.~Banaszczyk.
\newblock {The Steinitz Constant of the Plane}.
\newblock {\em Journal für die reine und angewandte Mathematik}, 373:218--220, 1987.

\bibitem{banaszczyk1990note}
W.~Banaszczyk.
\newblock {A note on the Steinitz constant of the Euclidean plane}.
\newblock {\em CR Math. Rep. Acad. Sci. Canada}, 12(4):97--102, 1990.

\bibitem{barany1981vector}
I.~B{\'a}rány.
\newblock {A Vector-sum Theorem and its Application to Improving Flow Shop Guarantees}.
\newblock {\em Mathematics of Operations Research}, 6(3):445--452, 1981.

\bibitem{Ba10}
I.~B{\'a}rány.
\newblock {On the Power of Linear Dependencies}.
\newblock In {\em Building bridges: Between Mathematics and Computer Science}, volume~19 of {\em Bolyai Society Mathematical Studies}, pages 31--45. Bolyai Society and Springer, 2010.

\bibitem{barany2024matrix}
I.~B{\'a}rány.
\newblock {A matrix version of the Steinitz lemma}.
\newblock {\em Journal f{\"u}r die reine und angewandte Mathematik (Crelles Journal)}, 2024(809):261--267, 2024.

\bibitem{behrend1954steinitz}
F.~Behrend.
\newblock {The Steinitz-Gross theorem on sums of vectors}.
\newblock {\em Canadian Journal of Mathematics}, 6:108--124, 1954.

\bibitem{BelovStolin}
I.~Belov and J.~Stolin.
\newblock {An Algorithm for the Flow Shop Problem}.
\newblock {\em Mathematical Economics and Functional Analysis}, 499, 1974.

\bibitem{bergstrom1931neuer}
V.~Bergstr{\"o}m.
\newblock {Ein neuer Beweis eines Satzes von E. Steinitz}.
\newblock In {\em Abhandlungen aus dem Mathematischen Seminar der Universit{\"a}t Hamburg}, volume~8, pages 148--152. Springer, 1931.

\bibitem{bergstrom1931Zwei}
V.~Bergstr{\"o}m.
\newblock {Zwei S\"atze \"uber ebene Vectorpolygone}.
\newblock In {\em Abhandlungen aus dem Mathematischen Seminar der Universit{\"a}t Hamburg}, volume~8, pages 206--214. Springer, 1931.

\bibitem{chobanyan}
S.~Chobanyan.
\newblock {Convergence A.S. of Rearranged Random Series in Banach Space and Associated Inequalities}.
\newblock In J.~Hoffmann-J{\o}rgensen, J.~Kuelbs, and M.~B. Marcus, editors, {\em Probability in Banach Spaces, 9}, pages 3--29, Boston, MA, 1994. Birkh{\"a}user Boston.

\bibitem{chobanyan1987structure}
S.~A. Chobanyan.
\newblock {Structure of the set of sums of a conditionally convergent series in a normed space}.
\newblock {\em Mathematics of the USSR-Sbornik}, 56(1):49, 1987.

\bibitem{damsteeg1950steinitz}
I.~Damsteeg and I.~Halperin.
\newblock {The Steinitz-Gross theorem on sums of vectors}.
\newblock {\em Trans. Roy. Soc. Canada Sect. III}, 44:31--35, 1950.

\bibitem{FialaSteinitz}
T.~Fiala.
\newblock {An approximate algorithm for the three machine problem}.
\newblock {\em Alkalmazott Matematikai Lapok}, 3:389--398, 1977.
\newblock (in Hungarian).

\bibitem{grinberg_sevastyanov_ValueOfSteinitzConstant}
V.~Grinberg and S.~Sevast'yanov.
\newblock {Value of the Steinitz Constant}.
\newblock {\em Functional Analysis and its Applications}, 14:125--126, 1980.

\bibitem{Gross1917}
W.~Gro{\ss}.
\newblock {Bedingt konvergente Reihen}.
\newblock {\em Monatshefte f{\"u}r Mathematik und Physik}, 28:221--237, 1917.

\bibitem{hadwiger1936satz}
H.~Hadwiger.
\newblock {Ein Satz {\"u}ber geschlossene Vektorpolygone des Hilbertschen Raumes}.
\newblock {\em Mathematische Zeitschrift}, 41(1):732--738, 1936.

\bibitem{hadwiger1940umordnungsproblem}
H.~Hadwiger.
\newblock { {\"U}ber das Umordnungsproblem im Hilbertschen Raum}.
\newblock {\em Mathematische Zeitschrift}, 46(1):70--79, 1940.

\bibitem{hadwiger1942satz}
H.~Hadwiger.
\newblock {Ein Satz {\"u}ber bedingt konvergente Vektorreihen}.
\newblock {\em Mathematische Zeitschrift}, 47(1):663--668, 1942.

\bibitem{halperin1986sums}
I.~Halperin.
\newblock {Sums of a series, permitting rearrangements}.
\newblock {\em CR Math. Rep. Acad. Sci. Canada}, 8(2):87--102, 1986.

\bibitem{halperin1989bibliography}
I.~Halperin and T.~Ando.
\newblock {Bibliography: Series of vectors and Riemann sums}.
\newblock {\em Hokkaido University}, 1989.

\bibitem{Kadets}
M.~Kadets.
\newblock {On a Property of Broken Lines in $n$-Dimensional Space}.
\newblock {\em Uspekhi Mat. Nauk.}, 8(1):139--143, 1953.

\bibitem{katznelson1974conditionally}
Y.~Katznelson and O.~C. McGehee.
\newblock {Conditionally convergent series in $\mathbb{R}^{\infty}$}.
\newblock {\em Michigan Mathematical Journal}, 21(2):97--106, 1974.

\bibitem{levy1905series}
P.~L{\'e}vy.
\newblock {Sur les s{\'e}ries semi-convergentes}.
\newblock {\em Nouvelles annales de Math{\'e}matiques: journal des candidats aux {\'E}coles Polytechnique et Normale}, 5:506--511, 1905.

\bibitem{oertel2024colorful}
T.~Oertel, J.~Paat, and R.~Weismantel.
\newblock {A Colorful Steinitz Lemma with Application to Block-Structured Integer Programs}.
\newblock {\em Mathematical Programming}, 204(1):677--702, 2024.

\bibitem{RiemannRearrangementThm}
B.~Riemann.
\newblock {\em {Über die Darstellbarkeit einer Function durch eine trigonometrische Reihe}}, volume~13.
\newblock Abhandlungen der Königlichen Gesellschaft der Wissenschaften zu Göttingen, 1868.

\bibitem{rosenthal1987remarkable}
P.~Rosenthal.
\newblock {The Remarkable Theorem of L{\'e}vy and Steinitz}.
\newblock {\em The American Mathematical Monthly}, 94(4):342--351, 1987.

\bibitem{SevSurvey1998}
S.~Sevastianov.
\newblock Nonstrict vector summation in multi-operation scheduling.
\newblock {\em Annals of Operations Research}, 83(0):179--212, 1998.

\bibitem{SevSurvey94}
S.~Sevast'janov.
\newblock On some geometric methods in scheduling theory: a survey.
\newblock {\em Discrete Applied Mathematics}, 55(1):59--82, 1994.

\bibitem{BanSevMaxNorm}
S.~Sevast'janov and W.~Banaszczyk.
\newblock {To the Steinitz Lemma in coordinate form}.
\newblock {\em Discrete Mathematics}, 169(1):145--152, 1997.

\bibitem{Sevastyanov74}
S.~V. Sevastyanov.
\newblock {Asymptotical Approach to Some Scheduling Problems}.
\newblock In {\em Third All-Union Conf. Problems of Theoretical Cybernetics}, pages 67--69, Thes. Dokl., Novosibirsk, 1974.
\newblock (in Russian).

\bibitem{Sevastyanov75}
S.~V. Sevastyanov.
\newblock {Asymptotical Approach to Some Scheduling Problems}.
\newblock {\em Upravlyaemye Sistemy}, 14:40--51, 1975.
\newblock (in Russian).

\bibitem{SevastyanovSteinitz}
S.~V. Sevastyanov.
\newblock {On the Approximate Solution of Some Problems of Scheduling Theory}.
\newblock {\em Metody Diskretnogo Analiza}, 32:66--75, 1978.
\newblock (in Russian).

\bibitem{Sevastyanov80}
S.~V. Sevastyanov.
\newblock {Approximation Algorithms for Johnson's and Vector Summation Problems}.
\newblock {\em Upravlyaemye Sistemy}, 20:64--73, 1980.
\newblock (in Russian).

\bibitem{Sevastyanov91}
S.~V. Sevastyanov.
\newblock {On a Compact Vector Summation}.
\newblock {\em Diskretnaya Matematika}, 3:66--72, 1991.
\newblock (in Russian).

\bibitem{Steinitz1913}
E.~Steinitz.
\newblock {Bedingt Konvergente Reihen und Konvexe Systeme}.
\newblock {\em Journal für die reine und angewandte Mathematik}, 143:128--176, 1913.

\bibitem{Steinitz1914}
E.~Steinitz.
\newblock {Bedingt Konvergente Reihen und Konvexe Systeme}.
\newblock {\em Journal für die reine und angewandte Mathematik}, 144:1--40, 1914.

\bibitem{Steinitz1916}
E.~Steinitz.
\newblock {Bedingt Konvergente Reihen und Konvexe Systeme}.
\newblock {\em Journal für die reine und angewandte Mathematik}, 146:1--52, 1916.

\end{thebibliography}

\bigskip

\noindent
{\sc Gergely Ambrus}
\smallskip

\noindent
{\small 
{\em Bolyai Institute, University of Szeged, Hungary \\ and\\ 
 Alfréd Rényi Institute of Mathematics, Budapest, Hungary
 }}
\smallskip

\noindent
e-mail address: \texttt{ambrus@server.math.u-szeged.hu}

\bigskip

\noindent
{\sc Rainie Heck}
\smallskip

\noindent
{\em Alfréd Rényi Institute of Mathematics, Budapest, Hungary\\ and\\
Bolyai Institute, University of Szeged, Hungary  
  }
\smallskip

\noindent
e-mail address: \texttt{rainie@renyi.hu}
\end{document}